\documentclass[11pt]{article}
\usepackage{epigamath}

\usepackage[notext]{kpfonts}
\usepackage{baskervald}

\setpapertype{A4}

\usepackage[french]{babel}

\usepackage[dvips]{graphicx}     

\title{\vspace{-0.3cm}Troisi\`eme groupe de cohomologie non ramifi\'ee d'un solide cubique sur un corps de fonctions d'une variable}
\titlemark{Solides cubiques}
\author{\vspace{0cm} Jean-Louis Colliot-Th\'el\`ene et Alena Pirutka}
\authoraddresses{
\authordata{J.-L. Colliot-Th\'el\`ene}{\firstname{Jean-Louis} \lastname{Colliot-Th\'el\`ene}\\
\institution{Universit\'e Paris-Sud, CNRS, Paris-Saclay, Math\'ematiques, B\^atiment  307, 91405 Orsay Cedex, France}\\
\email{jlct@math.u-psud.fr}}\\
\authordata{A. Pirutka}{\firstname{Alena} \lastname{Pirutka}\\
\institution{Courant Institute, New York University, 251 Mercer Street,  New York,  NY 10012, USA \\
National Research University Higher School of Economics, Russian Federation}\\
\email{pirutka@cims.nyu.edu}}
}
\authormark{J.-L. Colliot-Th\'el\`ene et A. Pirutka}
\date{\vspace{-5ex}} 
\journal{\'Epijournal de G\'eom\'etrie Alg\'ebrique} 
\acceptation{Received by the Editors on September 22, 2017, and in final form
on September 11, 2018. \\ Accepted on November 9, 2018.}

\newcommand{\articlereference}{Volume 2 (2018), Article Nr. 13}

 \usepackage[all]{xy}

\allowdisplaybreaks

\newcommand\Ab{{\rm  Ab}}
 \newcommand\Gal{{\rm Gal}}
\newcommand{\PP}{{\mathbb P}}

\newcommand{\Z}{{\mathbb Z}}
\newcommand{\Q}{{\mathbb Q}}
\newcommand{\C}{{\mathbb C}}

\newcommand{\oi}{\hskip1mm {\buildrel \simeq \over \rightarrow} \hskip1mm}

\newcommand{\ovX}{{\overline X}}

\newcommand{\Br}{{\operatorname{Br  }}}

\newcommand{\Prym}{\operatorname{Prym}}
\newcommand{\Pic}{\operatorname{Pic}}
\newcommand{\cod}{\operatorname{cd}}

\newtheorem{theo}{Th\'{e}or\`{e}me}[section]
\newtheorem{prop}[theo]{Proposition}

\newtheorem{lem}[theo]{Lemme}

\newtheorem{rema}[theo]{Remarque}

\begin{document}


\maketitle



\begin{prelims}

\vspace{-0.55cm}

\def\abstractname{R\'esum\'e}
\abstract{En combinant une m\'ethode de C. Voisin avec la descente galoisienne sur le groupe de Chow en codimension $2$, nous montrons que le
troisi\`eme groupe de cohomologie non ramifi\'ee d'un solide cubique lisse d\'efini sur le corps des fonctions d'une courbe complexe est nul.  Ceci implique que  la conjecture de Hodge enti\`ere pour les classes de degr\'e 4 vaut pour les vari\'et\'es projectives et  lisses de dimension 4 fibr\'ees en solides cubiques au-dessus d'une courbe, sans restriction sur les  fibres singuli\`eres.}

\motscles{Groupes de Chow ; cycles de codimension 2 ; cohomologie non ramifi\'ee ;
famille d'hypersurfaces cubiques ; jacobienne interm\'ediaire ; conjecture
de Hodge enti\`ere.}

\MSCclassfr{14C25; 14H40; 14C35; 14D06; 14Mxx; 14C30}

\vspace{0.15cm}

\languagesection{English}{%

\vspace{-0.05cm}
\textbf{Title. Third unramified cohomology group of a cubic threefold over a function field in one variable} \commentskip \textbf{Abstract.} We prove that the third unramified cohomology group of a smooth cubic threefold over the function field of a complex curve vanishes. For this, we combine a method of C. Voisin with Galois descent on the codimension $2$ Chow group. As a corollary, we show that the integral Hodge conjecture holds for degree $4$ classes on smooth projective fourfolds equipped with a fibration over a curve, the generic fibre  of which is a smooth cubic threefold, with arbitrary singularities on the special fibres.

\commentskip
}

\end{prelims}


\newpage

\setcounter{tocdepth}{1} \tableofcontents

 \section{Introduction}

 Soit $W$ une vari\'et\'e projective et lisse sur $\C$, le corps des complexes.
Pour tout entier $i\geq 1$, on dispose d'applications cycles 
$$ cl_{i} : CH^{i}(W) \to Hdg^{2i}(W,\Z(i)),$$
d\'efinies sur le groupe de Chow des cycles de codimension $i$
modulo l'\'equivalence rationnelle, \`a valeurs dans le groupe
\begin{equation*} 
 Hdg^{2i}(W,\Z(i)) \subset H^{2i}_{Betti}(W,\Z(i))
 \end{equation*}
image r\'eciproque du groupe des  classes de Hodge dans 
$H^{2i}_{Betti}(W,\Q(i))$.
L'application $cl_{1}$ est surjective.
Pour $i\geq 2$, la conjecture de Hodge pr\'edit que
le conoyau $Z^{2i}(W)$ de $cl_{i}$ est de torsion.

Dans \cite{CTV12}, on a montr\'e que si le groupe de Chow $CH_{0}(W)$ des z\'ero-cycles sur $W$
est ``support\'e'' sur une surface, alors le groupe  $Z^{4}(W)$ est fini et isomorphe
au groupe 
\begin{equation*}
H^3_{nr}(\C(W)/\C, \Q/\Z(2)) \subset H^3(\C(W),\Q/\Z(2))
\end{equation*}
 form\'e des classes
non ramifi\'ees dans le troisi\`eme groupe de cohomologie galoisienne 
du corps des fonctions de $W$, \`a valeurs dans $\Q/\Z(2)$.
Les groupes $Z^{4}(W)$ et $H^3_{nr}(\C(W)/\C,\Q/\Z(2))$  sont des invariants birationnels.

Ceci a permis de donner des exemples de $W$ rationnellement connexes
de toute dimension $d\geq 6$ avec $Z^{4}(W)\neq 0$.
Pour $d=4,5$ et $W$ rationnellement connexe,
 c'est une question ouverte si  l'on a $Z^{4}(W)= 0$\footnote{Depuis la soumission de notre article, S. Schreieder a  construit
de telles vari\'et\'es $W$ en dimension $d=4, 5$ avec $Z^4(W)\neq 0$.}.
 C'est connu pour $W$ une hypersurface cubique lisse dans $\mathbb P^5_{\C}$ \cite{Vo07}.
 
 Plus g\'en\'eralement, C.~Voisin \cite{Vo13} a montr\'e que  $Z^{4}(W)= 0$
pour toute $W$ de dimension $4$ munie d'un morphisme 
$f: W \to \Gamma$ vers une courbe complexe, projective et lisse $\Gamma$, de fibre g\'en\'erique $X/{\C(\Gamma)}$ 
 une hypersurface cubique dans $\mathbb P^4_{\C(\Gamma)}$, sous l'hypoth\`ese que les
seules singularit\'es des fibres de $f$ sont des singularit\'es quadratiques ordinaires.

Nous montrons ici que l'\'enonc\'e vaut sans cette restriction sur les fibres singuli\`eres.

\begin{theo}\label{ThmHIF}
Soit \,$\Gamma$ \,une \,courbe \,complexe, \,projective, \,lisse, \,connexe, \,et \,soit \,$f: \mathcal X\to \Gamma$ \,un \,morphisme \,propre et surjectif, de dimension relative $3$,  avec $\mathcal X$ lisse 
 connexe, 
 \`a fibre g\'en\'erique un solide cubique lisse.  
 Alors $H^3_{nr}(\C(\mathcal X)/\C, \Q/ \Z(2))=0$
 et la conjecture de Hodge enti\`ere vaut pour les cycles de codimension $2$ sur  $\mathcal X$.
\end{theo}

Le groupe $H^3_{nr}(\C(\mathcal X)/\C, \Q/\Z(2))$ est form\'e des \'el\'ements de $H^3(\C(\mathcal X),\Q/\Z(2))$ non ramifi\'es par rapport  \`a toutes les valuations. C'est donc un sous-groupe du groupe $H^3_{nr}(\C(\mathcal X)/\C(\Gamma), \Q/\Z(2))$ form\'e des classes non ramifi\'ees par rapport aux valuations triviales sur $\C(\Gamma)$. Nous montrons que ce groupe est  d\'ej\`a nul. Plus pr\'ecis\'ement, 
 soient
  $X$  la fibre g\'en\'erique de $f$, $k=\C(\Gamma)$,  $\bar k$ une cl\^oture alg\'ebrique de $k$, $G={\rm Gal}(\bar k/k)$ le groupe de Galois absolu et $\ovX=X\times_k \bar k$.
Utilisant des r\'esultats de $K$-th\'eorie alg\'ebrique et de cohomologie motivique \`a coefficients $\Z(2)$ \cite{CT15, Ka96}
on se ram\`ene  (voir le \S \ref{section-4})
 \`a \'etablir que l'application 
\begin{equation*}
 CH^2(X) \to CH^2(\ovX)^G
 \end{equation*}
est surjective.
Autrement dit, il faut voir que toute classe dans $CH^2(\ovX)^G$
 est l'image de la classe d'une courbe (d\'efinie sur $k$) trac\'ee sur $X$.

\begin{theo}\label{thmh3nr}
Soit $k$   un corps de fonctions d'une variable sur $\C$ et soit $X\subset \PP^4_k$ un solide cubique lisse.   Soient $\bar k$ une cl\^oture alg\'ebrique de $k$, $G={\rm Gal}(\bar k/k)$ le groupe de Galois absolu et $\ovX=X\times_k \bar k$. Alors :
\begin{itemize}
\item[\rm (i)] L'application naturelle $CH^2(X)\to CH^2(\ovX)^G$ est surjective;
\item[\rm (ii)] $H^3_{nr}(k(X)/k, \Q/\Z(2))=0$.
\end{itemize}
\end{theo}

 Le groupe ab\'elien $CH^2(\ovX)$ est  bien compris. C'est une extension de
 $\Z$ par le groupe des $\overline{k}$-points de la jacobienne interm\'ediaire de $\ovX$. 
Plus pr\'ecis\'ement, puisque $\ovX$ est rationnellement connexe, l'application degr\'e $CH^0(\ovX)\to \Z$ est un isomorphisme. Un argument de d\'ecomposition de la diagonale (\cite{BS83} Thm. 1(ii), cela utilise aussi le th\'eor\`eme de Merkurjev-Suslin \cite{MS}) montre que \'equivalence alg\'ebrique et \'equivalence homologique
sur les cycles de codimension $2$ sur $\ovX$ co\"incident. 
 On a ainsi une suite exacte
 \begin{equation*}
 0 \to  CH^2(\ovX)_{alg} \to CH^2(\ovX) \to \Z \to 0,
 \end{equation*}
o\`u le  groupe $CH^2(\ovX)_{alg}$ des classes de cycles alg\'ebriquement \'equivalents \`a z\'ero
est une incarnation des $\bar k$-points de la jacobienne interm\'ediaire, qui est a priori d\'efinie sur $\bar k$,
par exemple, apr\`es  le choix d'une $\bar k$-droite sur $\ovX$, comme la vari\'et\'e de Prym attach\'ee \`a la fibration
en coniques d\'etermin\'ee par cette droite 
(voir \cite[Th\'eor\`eme 2.1(iii), Th\'eor\`eme 3.1]{Beau77}). 
Pour comprendre l'action galoisienne, nous utilisons une autre incarnation : d'apr\`es Murre \cite{Mu74}, le groupe $CH^2(\ovX)_{alg}$
s'identifie au groupe des $\bar k$-points de la vari\'et\'e de Picard ${\bf Pic}^0_{S/k}$, not\'ee ici $J/k$, de la surface de Fano $S/k$  des droites trac\'ees sur $X$.   

Le groupe $CH^2(\ovX)$  est la r\'eunion disjointe des ensembles $CH^2(\ovX)_{d}$ form\'es des classes de cycles
dont l'intersection avec un hyperplan est de degr\'e 
$d \in \Z$. 
On a 
$CH^2(\ovX)_{d}=J_{d}({\bar k})$,
o\`u $J_{d}/k$ est un espace principal homog\`ene de la $k$-vari\'et\'e ab\'elienne $J$.
La section \ref{sJd} est consacr\'ee \`a la construction de ces espaces.

Pour prouver que, pour tout entier $d$, l'application 
\begin{equation}\label{surjectionvoulu}
 CH^2(X)_{d} \to CH^2(\ovX)_{d}^G
 \end{equation}
est surjective,
par un argument simple \cite[p.~164]{Vo13}, il suffit de le faire pour $d=5$ et $d=6$.

 Pour $d=5$ et $d=6$, sur le corps des complexes et $X \subset \mathbb P^4_{\C}$ une cubique {\it g\'en\'erale},
 des arguments de g\'eom\'etrie projective \'elabor\'es (Iliev, Markushevich, 
 Tikhomirov \cite{IM00, MT01} pour $d=5$, Voisin \cite{Vo13} pour $d=6$) ont montr\'e l'existence de d\'esin\-gu\-la\-risations projectives et lisses $\tilde{M}_{d,1}$ de certaines
 composantes du sch\'ema de Hilbert de  courbes de genre $1$ et de degr\'e $d$ trac\'ees sur $X$
 et de morphismes d'Abel-Jacobi  $\tilde{M}_{d,1} \to J_{d}$ qui sont dominants et   dont la fibre g\'en\'erique
 est une vari\'et\'e g\'eom\'etriquement rationnellement connexe. Dans la section \ref{sectionEM} on \'etend cette construction \`a la famille universelle de toutes les hypersurfaces cubiques lisses dans $\mathbb P^4_{\C}$.
 
 Pour {\it toute} hypersurface cubique lisse dans $ \mathbb P^4_{\C(\Gamma)}$ la surjection voulue (\ref{surjectionvoulu}) pour $d=5,6$ s'\'etablit par un argument de sp\'ecialisation (voir la Proposition \ref{pointsfibres}) pour les familles non-n\'ecessairement lisses de vari\'et\'es rationnellement connexes (\cite{GHMS05, HX09}) combin\'e au th\'eor\`eme de Graber, Harris et Starr \cite{GHS03}. La d\'emonstration des th\'eor\`emes \ref{ThmHIF} et \ref{thmh3nr} est faite dans la section \ref{preuvesT}.
 
\bigskip

{\bf Remerciements.} Ce travail a commenc\'e \`a Vienne pendant le programme 
``Advances in Birational Geometry 2017"; nous remercions l'Institut  Schr\"odinger pour son hospitalit\'e et Ludmil Katzarkov pour avoir lanc\'e ce programme.   Le  deuxi\`eme auteur remercie le Laboratoire de  Sym\'etrie Miroir  NRU HSE, RF Government grant, ag. 14.641.31.0001 et le NSF grant 1601680 pour leurs soutiens financiers. Le premier auteur remercie l'Institut Courant, NYU, pour son hospitalit\'e.

 \section{Repr\'esentants alg\'ebriques pour $CH^2_{alg}$}\label{sJd}

\subsection{Sur un corps alg\'ebriquement clos} \label{sac}

Soit $k$ un corps alg\'ebriquement clos et soit $X$ une vari\'et\'e alg\'ebrique
connexe,
 projective et lisse d\'efinie sur $k$.  Soit $A$ une vari\'et\'e ab\'elienne sur $k$. Rappelons \cite{Mu72, Mu73, Mu74}, \cite[Def. 1.6.1]{Mu85}, \cite[Section 3.2]{Beau77}, \cite{ACMV17} 
 qu'un 
 homomorphisme de groupes ab\'eliens
$\phi:CH^2(X)_{alg}\to A(k)$
est dit {\it r\'egulier} si pour toute donn\'ee d'une vari\'et\'e lisse 
connexe 
$T$ sur $k$ avec un point $t_0\in T(k)$ et un cycle $Z\in CH^2(X\times T)$  l'application compos\'ee

\begin{equation*}
T(k)\to CH^2(X)_{alg}\to A(k),\; t\mapsto \phi(Z_t-Z_{t_0})
\end{equation*}
est induite par un morphisme de vari\'et\'es alg\'ebriques $T\to A$ (voir \cite[Section 10.1]{Fulton} pour la d\'efinition du cycle $Z_t$). Un {\it repr\'esentant} alg\'ebrique  
pour  $CH^2(X)_{alg}$ 
est une vari\'et\'e ab\'elienne $\Ab^2(X)$
 munie d'un homomorphisme 
r\'egulier $\phi_{Ab}:CH^2(X)_{alg}\to \Ab^2(X)(k)$ qui v\'erifie la propri\'et\'e universelle suivante : pour toute vari\'et\'e ab\'elienne $A$ et pour tout  morphisme r\'egulier $\phi:CH^2(X)_{alg}\to A(k)$ il existe un unique morphisme de vari\'et\'es ab\'eliennes
 $\psi: \Ab^2(X)\to A$ tel que le diagramme suivant
$$\xymatrix{CH^2(X)_{alg}\ar[rr]^{\phi_{Ab}}\ar[rd]^{\phi}& & \Ab^2(X)(k)\ar[ld]^{\psi}\\
&A(k)&
}
$$
soit commutatif.
 Murre \cite[Thm. 1.9]{Mu85},
 en utilisant des r\'esultats de H. Saito, Bloch et Ogus, et Merkurjev et Suslin,
  a \'etabli
  qu'un repr\'esentant alg\'ebrique pour $CH^2(X)_{alg}$
  existe pour toute vari\'et\'e alg\'ebrique 
 connexe,
 projective et lisse sur $k$ (voir \cite[Thm. 5]{Mu74} pour les solides cubiques).
 D'apr\`es la d\'efinition, il est donc unique \`a un unique isomorphisme pr\`es.

 Dans le cas o\`u $X$ est un solide cubique lisse, on dispose d'une description explicite du repr\'esentant alg\'ebrique pour $CH^2(X)_{alg}$.  Supposons que $k$ est un corps  alg\'ebriquement clos de caract\'eristique diff\'erente de $2$
 et soit  $X\subset \PP^4_{k}$ un solide cubique lisse. 
Murre \cite{Mu72, Mu73, Mu74} (voir aussi  
Beauville \cite[Exemple 1.4.1 et Prop. 3.3]{Beau77})
exprime un repr\'esentant alg\'ebrique pour 
$CH^2(X)_{alg}$ 
en tant  
que  vari\'et\'e de Prym.
 Plus pr\'ecis\'ement, soit  $f:X'\to X$  l'\'eclatement de $X$ le long d'une droite $L$
suffisamment g\'en\'erale \cite[Prop. (1.25)]{Mu72}. On a alors $CH^2(X)_{alg}\simeq CH^2(X')_{alg}$.
En effet, si $E\subset X'$ est le diviseur exceptionnel, on a que $E\simeq \mathbb P^1_k\times \mathbb P^1_k$, car pour $L$ g\'en\'erale $N_{L/X}\simeq \mathcal O_L\oplus \mathcal O_L$ \cite[Prop. 6.19]{CG72}, et $CH^1(E)\simeq \mathbb Z\oplus \mathbb Z$. Soit $F$ la fibre de la restriction $f:E\to L=\mathbb P^1_k$. La formule d'\'eclatement \cite[Prop. 6.7(e)]{Fulton} donne alors un isomorphisme $ CH^2(X)\oplus \mathbb Z\simeq CH^2(X')$ o\`u l'inclusion du deuxi\`eme facteur  est donn\'ee par $n\mapsto n[F]$, d'o\`u l'isomorphisme $CH^2(X)_{alg}\simeq CH^2(X')_{alg}$.

Le solide  $X'$ poss\`ede une structure de fibr\'e en coniques $X' \to \PP^2$  ordinaire : la courbe de ramification $C \subset \PP^2$ est lisse (voir \cite[Prop. 1.22 et Section 5.1]{Mu72}).
 
 On a un rev\^{e}tement double \'etale $C' \to C$ avec $C'$ lisse connexe
 associ\'e aux syst\`emes de g\'en\'eratrices des fibres d\'eg\'en\'er\'ees.
 On d\'efinit la vari\'et\'e de Prym $\Prym(C'/C)$  du rev\^{e}tement connexe $C'/C$
 (voir \cite{Beau77}). C'est une vari\'et\'e ab\'elienne principalement polaris\'ee.
 On construit un isomorphisme de groupes ab\'eliens \cite[Thm. 3.1]{Beau77} :
 $$\Phi : \Prym(C'/C)(k) \oi CH^2(X')_{alg}.$$
 L'application inverse
 $$ \Phi^{-1} : CH^2(X')_{alg} \to \Prym(C'/C)(k)$$
 fait de $\Prym(C'/C)$ un repr\'esentant alg\'ebrique de $CH^2(X')_{alg}$ \cite [Prop. 3.3]{Beau77}.
 
 \subsection{Sur un corps de caract\'eristique z\'ero}

 Soit 
 $k$ un corps quelconque de caract\'eristique z\'ero,  $\bar k$ une cl\^oture alg\'ebrique de $k$
  et $G=\Gal(\bar k/k)$.  Soient $X\subset \PP^4_{k}$ un solide cubique lisse et $\ovX=X\times_k \bar k$. 
Dans la suite de ce texte nous avons besoin d'une autre description du repr\'esentant alg\'ebrique pour $CH^2(\ovX)_{alg}$.

Soit  $S$ le $k$-sch\'ema qui param\`etre les droites contenues dans $X$. C'est une 
$k$-surface
projective lisse 
g\'eom\'etriquement connexe, 
appel\'ee  {\it surface de Fano}  
 du solide cubique lisse
  $X$ \cite[1.12]{AK77}.  On note $\bar S=S\times_k \bar k$.
Soient $V\subset S\times X$
 \begin{equation}\label{vincidence}
V=\{(L,x), x\in L\} 
 \end{equation}
  la vari\'et\'e d'incidence associ\'ee et $p:V\to S, q:V\to X$ les deux projections.
  Soit  ${\bf Pic}_{S/k}$   le  
$k$-sch\'ema 
de Picard de $S$ : ce sch\'ema repr\'esente le faisceau  $\Pic_{S/k, (\acute{e}t)}$ associ\'e au foncteur $\Pic_{S/k}$ d\'efini par 
\begin{equation*}
\Pic_{S/k}(T)=\Pic(S\times_k T)/\Pic(T).
\end{equation*}
 Pour tout corps $K\supset k$  on dispose de l'application
 \begin{equation}\label{AJAlg}
p_*q^*: CH^2(X_K)\to {\bf Pic}_{S/k}(K)
 \end{equation}
induite par la correspondance d'incidence, o\`u   l'application $q^*$ est d\'efinie dans \cite[6.6]{Fulton} pour tout morphisme localement d'intersection compl\`ete, en particulier, pour tout morphisme entre des sch\'emas lisses (voir \cite[Ex. 6.3.18]{Liu}).  

L'application (\ref{AJAlg}) induit un homomorphisme $G$-\'equivariant
\begin{equation*}
p_*q^*: CH^2(\ovX) \to {\bf Pic}_{S/k}(\bar k)
\end{equation*}

Soit  $J:={\bf Pic}^0_{S/k}$ la composante connexe 
de l'identit\'e du sch\'ema  
${\bf Pic}_{S/k}$.  Par restriction,  on obtient   l'application $G$-\'equivariante
$$p_*q^*: CH^2(\ovX)_{alg}\to J(\bar k)={\bf Pic}^0_{S/k}(\bar k).$$

On dispose aussi de l'accouplement d'intersection
\begin{equation*}
CH^2(\ovX)  \times CH^1(\ovX) \to \Z
\end{equation*} qui
est $G$-\'equivariant. Le groupe $ CH^1(\ovX) = \Pic (\ovX) $
est \'egal \`a $\Z$, engendr\'e par la classe d'une section hyperplane
 \cite[Corollaire XII.3.7]{SGA2}.

\begin{theo}\label{ReprAlgd}
\begin{itemize}
\item[\rm (i)] (Murre) L'homomorphisme induit
\begin{equation}\label{AJbar}
CH^2(\ovX)_{alg} \to {\bf Pic}^0_{S/k}(\bar k)
\end{equation}
est  un isomorphisme de modules galoisiens.
 Il  fait de la $\bar k$-vari\'et\'e ab\'elienne   ${\bf Pic}_{S/k}^0\times_{k}\bar k$
  un repr\'esentant alg\'ebrique de $CH^2(\ovX)_{alg} $.  
\item[\rm (ii)] L'\'equivalence alg\'ebrique et l'\'equivalence num\'erique
co\"{\i}ncident sur $CH^2(\ovX)$;  elles co\"{\i}ncident aussi
avec l'\'equivalence homologique enti\`ere si $k=\mathbb C$.
On a une suite exacte de modules galoisiens
\begin{equation}\label{CHalg}
0 \to CH^2(\ovX)_{alg} \to CH^2(\ovX) \to \Z \to 0,
\end{equation}
o\`u la fl\`eche $CH^2(\ovX) \to \Z$ est donn\'ee par l'intersection
avec un hyperplan dans $\mathbb P^4_{\bar k}$.
\item[\rm (iii)] Pour tout 
nombre entier $d$ il existe une $k$-vari\'et\'e $J_{d}=J_d(X)$ qui est
un espace principal homog\`ene sous $J=J_{0}={\bf Pic}^0_{S/k}$ tel que $J_{d}(\bar k)$
s'identifie \`a l'image r\'eciproque de $d \in \Z$ via l'application $ CH^2(\ovX) \to \Z$ dans la suite (\ref{CHalg}) ci-dessus.
\item[\rm (iv)] Soit $T$ une $k$-vari\'et\'e lisse connexe et $Z \subset X \times_{k}T$
un ferm\'e plat sur $T$ dont les fibres sont des courbes de degr\'e $d$
sur $X$. L'application induite  $T(\bar k) \to CH^2(\ovX)$ 
est elle-m\^eme induite par un $k$-morphisme  $T \to  J_{d}$.
\end{itemize}  
\end{theo}
\begin{proof}
Pour (i) commen\c{c}ons par noter que l'homomorphisme (\ref{AJbar}) est r\'egulier : pour toute vari\'et\'e lisse connexe $T$ 
sur $ \bar k$,  
tout
point $t_0\in T(\bar k)$ et 
tout $Z\in CH^2(\ovX\times T)$, on dispose d'une famille de diviseurs  $p_*q^*Z\in CH^1(\bar S\times T)$. On a donc que l'application 
\begin{equation}\label{fcycles}
T(\bar k) \to {\bf Pic}^0_{S/k}(\bar k), \, t\mapsto p_*q^*(Z_t-Z_{t_0})
\end{equation}
 est 
 induite par
  un morphisme de vari\'et\'es alg\'ebriques $T\to {\bf Pic}^0_{S/k}$, car le sch\'ema ${\bf Pic}_{S/k}$ repr\'esente le faisceau  $\Pic_{S/k, (\acute{e}t)}$ associ\'e au foncteur $\Pic_{S/k}$ \cite{Kl05}.
  
  Soit $\Prym(C'/C)$ la vari\'et\'e de Prym associ\'ee \`a $\ovX$ comme dans la section \ref{sac} (apr\`es avoir choisi une droite g\'en\'erale contenue dans $\ovX$).  
Par la propri\'et\'e universelle, on a donc un 
 homomorphisme de vari\'et\'es ab\'eliennes
  $\Prym(C'/C)\to {\bf Pic}^0_{S/k}\times_{k}\bar k$, qui 
  induit
  un diagramme commutatif
$$\xymatrix{CH^2(\ovX)_{alg}\ar[rr]\ar[rd]_{p_*q^*}& & \Prym(C'/C)(\bar k)\ar[ld]\\
&{\bf Pic}^0_{S/k}(\bar k)&
}
$$
 D'apr\`es Murre \cite[Thm. 8]{Mu74}, l'application $\Prym(C'/C)\to {\bf Pic}^0_{S/k}\times_{k}\bar k$ est un isomorphisme.

Pour (ii), supposons d'abord $k=\mathbb C$.  La vari\'et\'e $X$ satisfait $CH_{0} (X_{\Omega})=\Z$ 
pour tout surcorps alg\'ebriquement clos $\Omega$ de $k$. D'apr\`es le th\'eor\`eme de Bloch et Srinivas \cite[Thm 1.2]{BS83}, qui utilise  le th\'eor\`eme de Merkurjev-Suslin, ceci implique que l'\'equivalence alg\'ebrique et l'\'equivalence homologique enti\`ere  co\"incident sur $CH^2(X)$.  L'\'equivalence homologique rationnelle et l'\'equivalence num\'erique co\"incident pour les $1$-cycles sur toute vari\'et\'e complexe connexe, projective et lisse (voir \cite[Prop. 1.1]{CTS13}). Pour le solide cubique $X$, la conjecture de Hodge enti\`ere vaut car $H^4(X, \mathbb Z)\simeq \Z$, engendr\'e par la classe d'une droite. En conclusion, les \'equivalences alg\'ebrique, num\'erique et homologique (enti\`ere et rationnelle) co\"incident sur $CH^2(X)$ et on a donc bien une suite exacte    
\begin{equation*}
0 \to CH^2(\ovX)_{alg} \to CH^2(\ovX) \to \Z \to 0
\end{equation*}
comme dans l'\'enonc\'e.
Dans le cas g\'en\'eral, fixons  un plongement $\bar k\subset \C$. Soit $\alpha\in CH^2(\ovX)$ une classe
num\'eriquement \'equivalente \`a z\'ero. On a le diagramme commutatif suivant
$$\xymatrix{CH^2(\ovX)\times CH^1(\ovX)\ar[r]\ar[d]&\Z\ar@{=}[d]\\
CH^2(X_{\C})\times CH^1(X_{\C})\ar[r]&\Z
}
$$
o\`u les fl\`eches horizontales sont induites par le produit d'intersection.
Puisque $CH^1(\ovX)\simeq CH^1(X_{\mathbb C})\simeq \Z$, on
 d\'eduit que l'image  $\alpha_{\C}$ de $\alpha$ dans $CH^2(X_{\C})$ est num\'eriquement, donc alg\'ebriquement, \'equivalente \`a $0$. Ainsi la classe $p_*q^*\alpha_{\C}\in{\bf Pic}_{S/k}(\C)$ est dans  ${\bf Pic}^0_{S/k}(\C)$. Or cette classe provient de ${\bf Pic}_{S/k}(\bar k)$. On a donc $p_*q^*\alpha \in{\bf Pic}^0_{S/k}(\bar k)$ et $\alpha$ est alg\'ebriquement \'equivalente \`a z\'ero. La suite exacte (\ref{CHalg}) en d\'ecoule.
 
Pour (iii), on utilise que l'extension (\ref{CHalg}) induit une suite exacte
\begin{equation*}
0\to CH^2(\ovX)_{alg}^G\to CH^2(\ovX)^G\to \Z\stackrel{\delta}{\to} H^1(k,CH^2(\ovX)_{alg}),
\end{equation*}
d'o\`u une classe de cohomologie $\delta(1)$  dans $H^1(k,CH^2(\ovX)_{alg} )= H^1(k, J(\bar k))$
et donc un espace principal homog\`ene $J_{1}$ de $J={\bf Pic}^0_{S/k}$.
  Pour $d \in \Z$ on dispose d'un espace
principal homog\`ene $J_d$ de classe $\delta(d)$ :  l'ensemble des $\bar k$-points de la $k$-vari\'et\'e $J_{d}$ est la classe \`a gauche de $d\in \Z$, i.e. l'image r\'eciproque de $d$ dans $ CH^2(\ovX)$, avec l'action du groupe de Galois $G$ induite par celle sur $ CH^2(\ovX)$. L'application $p_*q^*$  induit une application $G$-\'equivariante $J_d(\bar k)\hookrightarrow {\bf Pic}_{S/k}(\bar k)$, via le diagramme suivant   
\begin{equation}\label{diagJd}
\xymatrix{0\ar[r]&CH^2(\ovX)_{alg}\ar[r]\ar_{p_*q^*}^{\simeq}[d]&CH^2(\ovX)\ar[r]\ar[d]_{p_*q^*} &\Z\ar@{^{(}->}_{p_*q^*}[d]\ar[r]&0\\
0\ar[r]& {\bf Pic}^0_{S/k}(\bar k)\ar[r]& {\bf Pic}_{S/k}(\bar k)\ar[r] &{\bf NS}_{S/k}(\bar k)\ar[r]& 0.
}
\end{equation}
On peut donc aussi voir $J_d$ comme une composante du sch\'ema ${\bf Pic}_{S/k}$.

Soit  $T$ une $k$-vari\'et\'e lisse connexe et $Z \subset X \times_{k}T$ comme dans (iv). Soit $[Z]\in CH^2(X\times T)$ la classe de $[Z]$. La classe $p_*q^*[Z]\in CH^1( S\times_k T)$ induit un morphisme alg\'ebrique 
\begin{equation}\label{TJalg}
T\to {\bf Pic}_{S/k}.
\end{equation}
 Puisque les fibres de $Z$ au-dessus de $T$ sont des courbes de degr\'e $d$, l'image de $T(\bar k)$ est contenue dans $J_d(\bar k)\subset {\bf Pic}_{S/k}(\bar k)$. L'application (\ref{TJalg}) ci-dessus se factorise donc par une application alg\'ebrique $T\to J_d$, comme affirm\'e dans l'\'enonc\'e.
\qed
\end{proof}

\begin{rema}\label{classe}{\rm 
Dans le cas o\`u $k=\mathbb C$, la fl\`eche $CH^2(X)\to \Z$ dans l'extension (\ref{CHalg}) s'identifie \`a l'application classe de cycle $CH^2(X)\to H^4(X, \Z)$, o\`u $H^4(X, \Z)\simeq \Z$ est engendr\'e par la classe d'une droite. Via le diagramme (\ref{diagJd}) la vari\'et\'e $J_d$ correspond \`a une composante de  ${\bf Pic}_{S/k}$ d'\'el\'ements de classe $p_*q^*d$.}
\end{rema}

 \section{Espaces de modules pour une famille de cubiques}\label{sectionEM}
 
\subsection{Espaces de modules des courbes de genre $1$} Soit $k=\mathbb C$ et soit $X\subset \mathbb P^4_{\mathbb C}$ un solide cubique lisse.

Soit $C\subset \mathbb P^4_{\mathbb C}$ une courbe connexe, projective et lisse. Rappelons que $C$ est dite {\it non-d\'eg\'en\'er\'ee} si l'application
\begin{equation}\label{anormale}
H^0(\mathbb P^4_{\mathbb C}, \mathcal O_{\mathbb P^4_{\mathbb C}}(1))\to H^0(C, \mathcal O_{C}(1)),
\end{equation} 
induite par la suite exacte de faisceaux
\begin{equation*}
0\to I_C\to \mathcal O_{\mathbb P^4_{\mathbb C}}\to \mathcal O_{C}\to 0
\end{equation*}
est injective.
On dit que $C$ est {\it (lin\'eairement) normale} si l'application (\ref{anormale}) est surjective. On dit que $C$ est {\it projectivement normale} si l'application  $H^0(\mathbb P^4_{\mathbb C}, \mathcal O_{\mathbb P^4_{\mathbb C}}(m))\to H^0(C, \mathcal O_{C}(m))$ est surjective pour tout $m\geq 1$. Notons que la dimension de l'espace $H^0(\mathbb P^4_{\mathbb C}, \mathcal O_{\mathbb P^4_{\mathbb C}}(1))$ vaut $5$. Si $C$ est de genre $1$ et de degr\'e $5$, d'apr\`es Riemann-Roch, la dimension de l'espace $H^0(C, \mathcal O_{C}(1))$ vaut $5$ aussi. Dans ce cas, la courbe $C$ est non-d\'eg\'en\'er\'ee si et seulement si elle est lin\'eairement normale. D'apr\`es \cite[Prop. IV.1.2]{Hu86}, la courbe $C$ est alors aussi projectivement normale.

Dans la suite de ce texte, on aura besoin d'espaces de modules de courbes $C\subset X$ de genre $1$ et de degr\'e $5$ ou $6$. Suivant \cite{MT01, IM00, Vo13}, pour $d=5, 6$ on consid\`ere $M_{d,1}(X)$ l'union des composantes, avec la structure r\'eduite, du sch\'ema de Hilbert de $X$
 dont le point g\'en\'eral param\`etre les courbes $C \subset X$,
 lisses connexes, de genre $1$,  de degr\'e $d$, contenues dans $X$ et  non-d\'eg\'en\'er\'ees dans  $\mathbb P^4_{\mathbb C}$. Les points de ces espaces  $M_{5,1}(X)$ et  $M_{6,1}(X)$  param\`etrent donc  des courbes projectives (peut-\^etre r\'eductibles) de degr\'e $5$ et $6$ respectivement.
L'espace  $M_{5,1}(X)$ est irr\'eductible, de dimension $10$
   \cite[Thm. 4.5]{MT01},   \cite[Thm. 8.1]{HRS1}.
 Pour $X$  g\'en\'erale, 
l'espace $M_{6,1}(X)$ est irr\'eductible  de dimension 12 \cite[p.~153]{Vo13}.

On fixe  une d\'esingularisation $\tau_5:\tilde M_{5,1}(X) \to M_{5,1}(X)$ (resp. $\tau_6:\tilde M_{6,1}(X) \to M_{6,1}(X)$).  Les morphismes $\tau_5$ et $\tau_6$ sont des morphismes propres birationnels.

Si $Z_d\subset M_{d,1}(X)\times X$, $d=5,6$ est la restriction de la famille universelle, d'apr\`es le Th\'eor\`eme \ref{ReprAlgd}\,(iv), on dispose d'une application rationnelle $M_{d,1}(X)\dashrightarrow J_d$.
Puisque $\tilde M_{d,1}(X)$ est lisse,  cette application s'\'etend  en 
un morphisme $\tilde M_{d,1}(X)\to J_d$, d\'efini \`a partir de la famille $Z_d\times_{M_{d,1}(X)} \tilde{M}_{d,1}(X)$. On l'appelle ici l'application {\it d'Abel-Jacobi}.

\begin{theo}{\rm ({\cite[Thm. 3.2]{IM00}, \cite[Thm. 5.6]{MT01},
\cite[Thm. 4.7]{Mb16},
 \cite[Thm. 3.1]{Vo13}})}\label{fibresRC}
Soient $k=\C$ et $X\subset \PP^4_k$ un solide  cubique tr\`es g\'en\'eral.
Alors les applications d'Abel-Jacobi 
\begin{equation*}
\tilde M_{5,1}(X)\to J_5\mbox{ et } \tilde M_{6,1}(X)\to J_6
\end{equation*}
d\'efinies ci-dessus sont surjectives, \`a
  fibre g\'en\'erique g\'eom\'etrique 
rationnellement connexe. 
\end{theo}

La d\'emonstration de ces r\'esultats repose sur des arguments de g\'eom\'etrie projective \'elabor\'es.  Les arguments dans \cite{IM00, MT01, Vo13}  utilisent l'hypoth\`ese $k=\C$.  Dans le cas de l'espace de modules $\tilde M_{5,1}(X)$  l'\'enonc\'e vaut  pour $k$  alg\'ebriquement clos de caract\'eristique positive, diff\'erente de $2$, d'apr\`es \cite{Mb16}.  Dans ce cas, on a un \'enonc\'e plus fort :
sur les $k$-points  g\'en\'eraux, les fibres  de l'application d'Abel-Jacobi sont isomorphes \`a $\PP^5_{k}$.

\subsection{Construction en famille}\label{sfamille} Soit $P\subset \PP^{34}_{\C}$ l'ouvert, dans l'espace projectif des coefficients des solides cubiques, qui param\`etre les solides lisses et soit  $\mathcal X\to P$ la famille universelle correspondante. On dispose d'une famille $\mathcal S\to P$ dont la fibre en un point $t\in P$ est la surface de Fano de $\mathcal X_{t}$. On dispose aussi d'une correspondance d'incidence $\mathcal V\subset \mathcal S\times \mathcal X$ et on note comme avant $p:\mathcal V\to \mathcal S$ et $q:\mathcal V\to \mathcal X$ les deux projections:
\begin{equation}\label{dP}\xymatrix{
 &\mathcal V\ar[dd]\ar[ld]_p\ar[rd]^q&  &\\
\mathcal S\ar[rd]&  & \mathcal X\ar[ld]&\\
&  P&  &.
}
\end{equation}
Dans ce diagramme, les morphismes $p$ et $\mathcal S\to  P$ sont lisses, et donc le morphisme $\mathcal V\to  P$ l'est aussi. Les sch\'emas quasi-projectifs $\mathcal V, \mathcal X, P$ sont des $\mathbb C$-sch\'emas lisses. Le morphisme $q$ est un morphisme localement d'intersection compl\`ete, puisque les sch\'emas $\mathcal V$ et $\mathcal X$ sont lisses sur $\mathbb C$.

Soit  $\mathcal M_{d,1}$, $d=5,6$  la composante  irr\'eductible (avec sa structure r\'eduite) du sch\'ema de Hilbert  relatif $Hilb_{\mathcal X/P}$ dont le point complexe g\'en\'eral param\`etre 
  une courbe lisse connexe de genre $1$ et de degr\'e $d$, non d\'eg\'en\'er\'ee dans $\mathbb P^4_{\mathbb C}$, et qui  domine $ P$. Une telle composante existe et elle est unique puisque   l'espace $M_{d,1}(\mathcal X_t)$ est irr\'eductible pour $t\in P(\mathbb C)$ g\'en\'eral, de dimension $10$ pour $d=5$ et de dimension $12$ pour $d=6$ (voir  la section pr\'ec\'edente). On a que $\mathcal M_{d,1}$ est un sous-sch\'ema ferm\'e du sch\'ema de Hilbert relatif $Hilb_{\mathcal X/P}$, et que $\mathcal M_{d,1,t}\subset M_{d,1}(\mathcal X_t)$ pour $t\in  P(\mathbb C)$. On fixe $\tau_d:\tilde{ \mathcal M}_{d,1}\to \mathcal M_{d,1}$ une d\'esingularisation de $\mathcal M_{d,1}$. Le morphisme $\tau_d$ est propre est birationnel.
 
On dispose d'une famille universelle $\mathcal Z_d\subset  \mathcal X \times_{ P}\mathcal M_{d,1}$, la projection $\mathcal Z_d\to \mathcal M_{d,1}$ est un morphisme 
plat, dont les fibres sont des courbes de degr\'e $d$ et de genre $1$.  On note $\tilde{\mathcal Z_d}=\mathcal Z_d\times_{\mathcal M_{d,1}}\tilde{\mathcal M}_{d,1}$ la famille induite, de sorte que le digramme suivant est un produit fibr\'e
 $$\xymatrix{
 \tilde{\mathcal Z}_d\ar@{^{(}->}[r]\ar[d]& \mathcal X\times_{ P} \tilde{\mathcal M}_{d,1} \ar[d]&\\
 \mathcal Z_d\ar@{^{(}->}[r]& \mathcal X \times_{ P}\mathcal M_{d,1} .&
 }
$$

On consid\`ere le diagramme suivant induit par le diagramme (\ref{dP}):
\begin{equation}\label{dP2}\xymatrix{
 &\mathcal V\times_P\tilde{\mathcal M}_{d,1}\ar[dd]\ar[ld]_{ p}\ar[rd]^{ q}&  &\\
\mathcal S\times_P\tilde{\mathcal M}_{d,1}\ar[rd]&  & \mathcal X\times_P\tilde{\mathcal M}_{d,1}\ar[ld]&\\
& \tilde{\mathcal M}_{d,1}&  &.
}
\end{equation}

Les morphismes $p$ et $\mathcal V\times_P\tilde{\mathcal M}_{d,1}\to \tilde{\mathcal M}_{d,1}$ sont lisses, les sch\'emas $\tilde{\mathcal M}_{d,1}$, $\mathcal V\times_P\tilde{\mathcal M}_{d,1}$ et $\mathcal X\times_P\tilde{\mathcal M}_{d,1}$ sont lisses aussi, cela implique en particulier que le morphisme $q$ est un morphisme localement d'intersection compl\`ete. On note 
\begin{equation}
q^*: CH_*(\mathcal X\times_P\tilde{\mathcal M}_{d,1})\to CH_*(\mathcal V\times_P\tilde{\mathcal M}_{d,1})
\end{equation}
le morphisme de Gysin "raffin\'e" qui est d\'efini pour tout morphisme localement d'intersection compl\`ete dans \cite[Chapter 6.6]{Fulton}. Puisque $\tilde p$ est propre, le morphisme 
\begin{equation*}
p_*: CH_*(\mathcal V\times_P\tilde{\mathcal M}_{d,1})\to CH_*(\mathcal S\times_P\tilde{\mathcal M}_{d,1})
\end{equation*}
 est bien d\'efini lui aussi. 
On peut donc d\'efinir la classe  $p_*q^*\tilde{\mathcal Z_d}$ dans 
\begin{equation*}
CH^1(\mathcal S\times_P\tilde{\mathcal M}_{d,1})=\Pic(\mathcal S\times_P\tilde{\mathcal M}_{d,1}),
\end{equation*} induite par la famille $\tilde{\mathcal Z_d}$; la derni\`ere \'egalit\'e vient du fait que le sch\'ema $S\times_P\tilde{\mathcal M}_{d,1}$ est lisse. 
 
 Soit  ${\bf Pic}_{\mathcal S/P}$   le  
$P$-sch\'ema 
de Picard de $\mathcal S$ : ce sch\'ema repr\'esente le faisceau  $\Pic_{\mathcal S/P, (\acute{e}t)}$ associ\'e au foncteur $\Pic_{\mathcal S/P}$ d\'efini par 
\begin{equation*}
\Pic_{\mathcal S/P}(T)=\Pic(\mathcal S\times_P T)/\Pic(T).
\end{equation*}
L'image de $p_*q^*\tilde{\mathcal Z_d}$ par la surjection naturelle 
$$ \Pic(\mathcal S\times_P\tilde{\mathcal M}_{d,1})\to  \Pic(\mathcal S\times_P\tilde{\mathcal M}_{d,1})/\Pic(\tilde{\mathcal M}_{d,1})$$ induit un \'el\'ement de ${\bf Pic}_{\mathcal S/P}(\tilde{\mathcal M}_{d,1})$. On a donc un morphisme 
\begin{equation}\label{AJdsPic}
\tilde{\mathcal M}_{d,1}\to {\bf Pic}_{\mathcal S/P}.
\end{equation}

\begin{lem}\label{AJfonctorialite}
Soit $a_1\in P$ un point et soit $F=\kappa(a_1)$ le corps r\'esiduel de $a_1$.  Soient $a_2\in \tilde{\mathcal M}_{d,1}(F)$ au-dessus du point $a_1$   et  $a_3=\tau_d(a_2)\in {\mathcal M}_{d,1}(F)$. On a alors
$$(p_*q^*\tilde{\mathcal Z}_{d})|_{a_2}=p_*q^*(\tilde {\mathcal Z}_{d, a_2})=p_*q^*(\mathcal Z_{d,a_3}),$$
o\`u $\tilde {\mathcal Z}_{d, a_2}$ (resp. $\mathcal Z_{d,a_3}$) est la fibre de $\tilde {\mathcal Z}_{d}$ (resp. $ {\mathcal Z}_{d}$) en un point $a_2$ (resp. $a_3$), le morphisme $p_*q^*$ est d\'efini ci-dessus (voir aussi (\ref{AJAlg})) et $(p_*q^*\tilde{\mathcal Z}_{d})|_{a_2}$ est la reststriction du cycle $(p_*q^*\tilde{\mathcal Z}_{d})$ au point $a_2$: l'image par le morphisme de Gysin associ\'e au morphisme $a_2\to \tilde{\mathcal M}_{d,1}$ entre les sch\'emas lisses.
\end{lem}
\begin{proof}
On note d'abord que, puisque $\tilde {\mathcal Z}_{d}$ est plat sur $\tilde{\mathcal M}_{d,1}$, on a que la fibre $\tilde {\mathcal Z}_{d, a_2}$ est une courbe de degr\'e $d$ dont la classe dans $CH_1(\mathcal X_{a_1})$ est bien d\'efinie. On a l'\'egalit\'e  $\tilde {\mathcal Z}_{d, a_2}=\mathcal Z_{d,a_3}$ d'apr\`es la d\'efinition. Supposons d'abord $F=\mathbb C$.  On consid\`ere le produit fibr\'e suivant
$$\xymatrix{
 \mathcal V_{a_1}\ar[r]\ar[d]^q& \mathcal V\times_P \tilde{\mathcal M}_{d,1}\ar[d]^q&\\
 \mathcal X_{a_1}\ar[r]& \mathcal X\times_P \tilde{\mathcal M}_{d,1}&
 }
$$
Tous les sch\'emas dans ce diagramme sont des sch\'emas lisses; toutes les fl\`eches sont donc des morphismes localement d'intersection compl\`ete. Par la fonctorialit\'e de la construction du morphisme de Gysin \cite[Theorem 6.5, Theorem 6.6]{Fulton} on a 
\begin{equation}\label{comp1}
q^*(\tilde {\mathcal Z}_{d, a_2})=(q^*\tilde{\mathcal Z}_{d})|_{a_2}
\end{equation} pour le cycle 
$\tilde {\mathcal Z}_{d, a_2}\in CH_1(\mathcal X_{a_1})$.
De m\^eme, on a le diagramme 
$$\xymatrix{
 \mathcal V_{a_1}\ar[r]\ar[d]^p& \mathcal V\times_P \tilde{\mathcal M}_{d,1}\ar[d]^p&\\
 \mathcal S_{a_1}\ar[r]& \mathcal S\times_P \tilde{\mathcal M}_{d,1}&
 }
$$
et, par la compatibilit\'e des morphismes $p_*$ et des morphismes de Gysin (c'est-\`a-dire des restrictions $|_{a_2}$) \cite[Theorem 6.2, Theorem 6.6]{Fulton} on obtient 
\begin{equation}\label{comp2}
p_*((q^*\tilde{\mathcal Z}_{d})|_{a_2})=(p_*q^*\tilde{\mathcal Z}_{d})|_{a_2}.
\end{equation} 
On d\'eduit donc des \'egalit\'es (\ref{comp1}) et (\ref{comp2}) le r\'esultat voulu:
\begin{equation*}
p_*q^*(\tilde {\mathcal Z}_{d, a_2})=p_*((q^*\tilde{\mathcal Z}_{d})|_{a_2})=(p_*q^*\tilde{\mathcal Z}_{d})|_{a_2}.
\end{equation*}

Le cas g\'en\'eral ($F$ est le corps r\'esiduel du point $a_1$) se montre par le m\^eme argument: on remplace $ \mathcal V_{a_1}$ par des sch\'emas lisses au-dessus d'un ouvert (suffisamment petit) dans l'adh\'erence du point $a_1$.
\qed
\end{proof}

Soit $\mathcal J\to P$  la famille de ``Jacobiennes interm\'ediaires"  $\mathcal J=\mathcal J_0={\bf Pic}^0_{ \mathcal S/P}$.

\begin{prop}\label{Jdenf}
Pour tout $d\in \Z$ il existe un sous-sch\'ema ferm\'e 
\begin{equation*}
\mathcal J_{d}\subset {\bf Pic}_{\mathcal S/P}
\end{equation*}
 tel que, pour tout $s\in P$, $\mathcal J_{d, s}=J_d(\mathcal X_s)$ comme d\'efini dans le th\'eor\`eme \ref{ReprAlgd}\,(iii).
\end{prop}
\begin{proof}
Pour tout $U\to P$ un sch\'ema lisse au-dessus de $P$ tel que $S(U)\neq \emptyset$ on d\'efinit un sous-sch\'ema $\mathcal J_{d, U}\subset {\bf Pic}_{\mathcal S/P}\times_P U$ comme le translat\'e du sch\'ema ${\bf Pic}^0_{\mathcal S/P}\times_P U$ par $dp_*q^*L_U$,
o\`u $L_U$ est une famille de droites qui correspond \`a un \'el\'ement de $S(U)$. Notons que cette d\'efinition ne d\'epend pas du choix de $L_U$, car pour une autre famille $L'_U$ on a que pour tout point g\'eom\'etrique $\bar u\in U$ les classes des droites $L_{\bar u}$ et $L'_{\bar u}$ sont alg\'ebriquement \'equivalentes dans $CH^2(\mathcal X_{\bar u})$, donc $dp_*q^*L_U-dp_*q^*L'_U$  induit un \'el\'ement de ${\bf Pic}_{\mathcal S/P}^0(U)$. On  en d\'eduit en particulier que pour tout $V\to U$ avec $V$ lisse, $\mathcal J_{d, V}=\mathcal J_{d, U}\times_U V.$

On choisit un recouvrement \'etale $\{U_i\to P\}_{i\in I}$ de $P$ tel que $\mathcal S(U_i)\neq \emptyset$ pour tout $i\in I$. D'apr\`es la construction ci-dessus, on dispose des sous-sch\'emas  $\mathcal J_{d, U_i}\subset {\bf Pic}_{\mathcal S/P}\times_P U_i$, tels que pour tout $i,j\in I$ on a $\mathcal J_{d,U_i}\times_{U_i} U_j=\mathcal J_{d,U_j}\times_{U_j} U_i$ comme  sous-sch\'emas de ${\bf Pic}_{\mathcal S/P}\times_P U_i\times_P U_j$. Par la descente  {\it fpqc} des sous-sch\'emas ferm\'es \cite[Chapter 6, Thm. 4]{BLR}, il existe un sous-sch\'ema ferm\'e $\mathcal J_d\subset {\bf Pic}_{\mathcal S/P}$ tel que $\mathcal J_{d, U_i}=\mathcal J_d\times_P U_i$ pour tout $i\in I$. 

D'apr\`es la construction, on a  $\mathcal J_{d, s}=J_d(\mathcal X_s)$  pour tout $s\in P$.
\qed
\end{proof}

Puisque $\tilde{\mathcal M}_{d,1}$ param\`etre des courbes de degr\'e $d$,
on a que
 le morphisme (\ref{AJdsPic}) se factorise par le sch\'ema $\mathcal J_{d}\subset {\bf Pic}_{\mathcal S/P}$ et induit donc un morphisme alg\'ebrique
\begin{equation}\label{AJM}
\phi_d: \tilde{\mathcal M}_{d,1}\to \mathcal J_d
\end{equation}
au-dessus de $P$.

\begin{prop}\label{pointsfibres}
Soit  $\mathcal X\to P$ la famille des solides cubiques lisses. 
Pour $d \in \{5,6\}$, soit
 $\phi_d: \tilde{\mathcal M}_{d,1}\to \mathcal J_d$ l'application d\'efinie ci-dessus. 

 Pour tout point $s\in \mathcal J_d$ de corps r\'esiduel $k=\kappa(s)$ le corps des fonctions d'une variable sur $\C$, la fibre $\tilde {\mathcal M}_{d,1, s}$ au-dessus de $s$ admet un $k$-point : $\tilde {\mathcal M}_{d,1, s}(k)\neq \emptyset.$ 
\end{prop} 
\begin{proof}
D'apr\`es la compatibilit\'e dans le lemme \ref{AJfonctorialite}, on peut appliquer le  th\'eor\`eme \ref{fibresRC} \`a la restriction  de l'application $\phi_d: \tilde{\mathcal M}_{d,1}\to \mathcal J_d$ \`a la fibre g\'en\'erique g\'eom\'etrique, qui est donc   rationnellement connexe. D'apr\`es un r\'esultat de Hogadi et Xu \cite[Thm. 1.2]{HX09} (voir aussi \cite[Prop. 2.7]{GHMS05}), toute fibre de $\phi_d$ au dessus d'un point  (pas n\'ecessairement un point ferm\'e) de $\mathcal J_d$  admet une sous-vari\'et\'e g\'eom\'etriquement irr\'eductible et rationnellement connexe. On applique ce r\'esultat \`a la fibre  $\tilde{{\mathcal M}}_{d,1, s}$ comme dans l'\'enonc\'e : soit $W\subset  \tilde{\mathcal M}_{d,1, s}$ une sous-vari\'et\'e rationnellement connexe, d\'efinie sur $k$. On a alors $W(k)\neq \emptyset$ d'apr\`es le th\'eor\`eme de Graber, Harris et Starr \cite{GHS03}, d'o\`u le r\'esultat.
\qed
\end{proof}

\section{Preuve des th\'eor\`emes  \ref{ThmHIF} et \ref{thmh3nr}}\label{preuvesT}

\label{section-4}

\begin{proof}[Preuve du th\'eor\`eme \ref{thmh3nr}] \

Montrons {\rm (i)}. Soit $\xi\in CH^2(\ovX)^G.$ Soit $h^2\in CH^2(X)$, o\`u $h$ est la classe d'une section hyperplane, le degr\'e de $h$ (la classe de $h$ dans  $H^4(\ovX, \Z(2))\simeq \Z$) vaut $3$. On peut donc supposer, quitte \`a remplacer $\xi$ par $\pm\xi+nh$, o\`u $n\in \Z$, que le degr\'e $d$ de $\xi$ est $d=5$ ou $6$.   Soit $\mathcal X\to P$ la famille universelle des cubiques, comme d\'efinie dans la partie \ref{sfamille}. Puisque la cubique $X$ est lisse, il existe un $k$-point $a_1\in P(k)$ tel que $X=\mathcal X_{t}$. Soit $s=p_*q^*(\xi)\in \mathcal J_{d}(\bar k)$ (cf. Th\'eor\`eme \ref{ReprAlgd} et Proposition \ref{Jdenf}).  
Puisque $\xi \in CH^2(\ovX)^G$ et l'application $p_*q^*$ est \'equivariante pour l'action du groupe $G$ d'apr\`es la d\'efinition (\ref{AJAlg}),
 on en d\'eduit que $s$ provient d'un point de  $\mathcal J_{d}(k)$ (que l'on note toujours $s$). D'apr\`es la proposition \ref{pointsfibres}, il existe un point $a_2\in \tilde{\mathcal M}_{d,1,s}(k)$ au-dessus de $s$.  
Par la fonctorialit\'e dans le lemme \ref{AJfonctorialite}  l'image $a_3$ de ce point 
 dans ${\mathcal M}_{d,1,s}(k)\subset Hilb\, \mathcal X_t(k)$ correspond \`a une courbe $C=\mathcal Z_{d, a_3}$ de genre $1$ de degr\'e $d$ sur $X=\mathcal X_{t}$ telle que $p_*q^*C=s=p_*q^*(\xi)$ et donc l'image de la classe de $C$ dans $CH^2(\ovX)^G$ vaut $\xi$.   
 On en d\'eduit  que l'application $CH^2(X)\to CH^2(\ovX)^G$ est surjective. 

Montrons maintenant  {\rm  (ii)}.  D'apr\`es \cite[Corollaire 4.2(iii)]{CT15}, 
qui s'applique pour $X$ un solide cubique lisse sur $k$  corps de  fonctions d'une variable sur $\C$, on a une suite exacte
\begin{multline*}
H^1(G, \mathrm{Pic}\, \ovX\otimes \bar k^*)\to H^3_{nr}(k(X)/k, \Q/\Z(2))/H^3(k, \Q/\Z(2))\\\to \mathrm{coker}[CH^2(X)\to CH^2(\ovX)^G]\to H^2(G, \mathrm{Pic}\, \ovX\otimes \bar k^*).
\end{multline*}   
Pour \'etablir l'existence de cette suite exacte, on utilise des techniques de $K$-th\'eorie et les groupes de cohomologie motivique \`a coefficients $\mathbb Z(2)$  \cite{Ka96}.

Par \cite[Corollaire XII.3.7]{SGA2}, on a $\mathrm{Pic}\, \ovX\simeq \Z$; 
 le premier groupe est nul par le th\'eor\`eme  $90$ de Hilbert.  On a aussi $H^3(k, \Q/\Z(2))=0$
  car $\cod(k)\leq 1$.
 On en d\'eduit qu'on a une injection $$H^3_{nr}(k(X)/k, \Q/\Z(2))\hookrightarrow \mathrm{coker}[CH^2(X)\to CH^2(\ovX)^G].$$  Le r\'esultat {\rm (i)} donne alors {\rm (ii)}. 
[
On remarque que le groupe $H^2(G, \mathrm{Pic}\, \ovX\otimes \bar k^*) = \Br\,k$ est nul lui aussi car $\cod(k)\leq 1$.]
\qed
\end{proof}

\begin{proof}[Preuve du th\'eor\`eme \ref{ThmHIF}] \

Le\,th\'eor\`eme\,\cite[\!Th\'eor\`eme\,1.1]{CTV12}\,s'applique\,\`a\,$\mathcal X$\,et\,donne\,un\,isomorphisme\,$H^3_{nr}(\C(\mathcal X)/\C, \Q/ \Z(2)) \oi  Z^4(\mathcal X).$ Il suffit donc de consid\'erer le premier groupe. Soient $k=\C(\Gamma)$  et $X/k$ la fibre g\'en\'erique de $f$. Puisque $H^3_{nr}(\C(\mathcal X)/\C, \Q/ \Z(2))\subset H^3_{nr}(k(X)/k, \Q/\Z(2))$, le th\'eor\`eme \ref{thmh3nr}\,(ii) implique   $H^3_{nr}(\C(\mathcal X)/\C, \Q/ \Z(2))=0$.  
\qed
\end{proof}

\begin{rema}{\rm
L'article \cite{Vo13} 
  contient l'analogue du th\'eor\`eme \ref{ThmHIF} pour les 
familles $\mathcal X \to \Gamma$ dont la fibre g\'en\'erique $X$ est une intersection lisse de deux quadriques dans $\PP^5$,
avec  une restriction sur les fibres sp\'eciales \cite[Thm. 1.4]{Vo13} (\cite[Cor. 2.7]{Vo13}).
 Dans ce cas, cette restriction a \'et\'e \'elimin\'ee
dans \cite[Cor. 3.1]{CT12}. Il suffisait  l\`a d'invoquer le calcul de la
cohomologie non ramifi\'ee en degr\'e 3 des quadriques sur un corps quelconque, qui donne $H^3_{nr}(k(X)/k, \Q/\Z(2))=0$ pour $k=\C(\Gamma)$ et donc $H^3_{nr}(\C(\mathcal X)/\C, \Q/\Z(2))=0$.}
 \end{rema}

\providecommand{\bysame}{\leavevmode\hbox to3em{\hrulefill}\thinspace}
%
%

\bibliographystyle{amsalpha}
\bibliographymark{References}
\def\cprime{$'$}

\end{document}